%% file: products-and-powers.tex
\documentclass[11pt]{amsart}
\usepackage[]{fontenc}
\usepackage[latin9]{inputenc}
\usepackage[letterpaper]{geometry}
\usepackage{amsthm}
\usepackage{amstext}
\usepackage{setspace}
\usepackage{amssymb}
\makeatletter
\numberwithin{equation}{section} 
\numberwithin{figure}{section} 
\theoremstyle{plain}
\newtheorem{thm}{Theorem}
  \theoremstyle{plain}
  \newtheorem{prop}[thm]{Proposition}
  \theoremstyle{plain}
  \newtheorem{lem}[thm]{Lemma}

\usepackage{color}
\usepackage{epsfig}

\usepackage{enumitem}

\makeatother

\begin{document}

\title{Inverses, Powers and Cartesian products of topologically deterministic
maps}

\author{Michael Hochman and Artur Siemaszko}
\begin{abstract}
We show that if $(X,T)$ is a topological dynamical system with is
deterministic in the sense of Kami\'{n}ski, Siemaszko and Szyma\'{n}ski
then $(X,T^{-1})$ and $(X\times X,T\times T)$ need not be determinstic
in this sense. However if $(X\times X,T\times T)$ is deterministic
then $(X,T^{n})$ is deterministic for all $n\in\mathbb{N}\setminus\{0\}$.
\end{abstract}


\thanks{M.H. supported by NSF grant 0901534. \\
A.S. supported by MNiSzW grant N201384834.}

\curraddr{Michael Hochman, Fine Hall, Washington Road, Princeton University,
Princeton, NJ 08544, USA }

\email{hochman@math.princeton.edu}

\curraddr{Artur Siemaszko, Faculty of Mathematics and Computer Science, University
of Warmia and Mazury in Olsztyn, ul. \.{Z}o\l{}nierska 14A, 10-561
Olsztyn, Poland}

\email{artur@uwm.edu.pl}

\maketitle

\section{Introduction}

By a topological dynamical system we mean a pair $(X,T)$, where $X$
is a compact metric space, and $T:X\rightarrow X$ an onto continuous
map. A factor map between systems $(X,T)$ and $(Y,S)$ is a continuous
onto map $\pi:X\rightarrow Y$ satisfying $S\pi=\pi T$. 

This note concerns systems $(X,T)$ which are topologically deterministic
(TD): i.e., whenever $(Y,S)$ is a factor of $(X,T)$, the map $S$
is invertable. This notion was introduced by Kami\'{n}ski, Siemaszko
and Szyma\'{n}ski in \cite{KaminskiSiemaszkoSzymanski03} as a natural
topological analogue of determinism in ergodic theory, which can be
defined similarly. Most work to date has focused on the relation of
TD and topological entropy, see \cite{KaminskiSiemaszkoSzymanski03,Hochman2010c}.
A relative version, analogous to the relative entropy theory, was
introduced in \cite{KaminskiSiemaszkoSzymanski2005}. Our purpose
here is to study some other basic properties of TD systems, namely,
the relation between determinism of $(X,T)$ and determinism of the
systems $(X,T^{n})$ and $(X\times X,T\times T)$. 

In the ergodic category, i.e. for measurable transformations $T$
preserving a probability measure $\mu$, the analogous notion of determinism
is that every measurable factor is invertible, and this is well-known
to be equivalent to the vanishing of the Kolmogorov-Sinai entropy.
Since $h(T^{n},\mu)=|n|h(T,\mu)$, $n\in\mathbb{Z}\setminus\{0\}$,
and $h(T\times T,\mu\times\mu)=2h(T,\mu)$, the vanishing of any one
of these implies the same for the others, and hence determinism of
$T$, $T^{n}$ and $T\times T$ are equivalent. In the topological
category, determinism is not equivalent to zero topological entropy,
and, as it turns out, the relation between determinism of powers and
products is more tenuous.
\begin{thm}
\label{thm:inverse}There exist TD systems $(X,T)$ such that $(X,T^{-1})$
is not TD. 
\end{thm}

\begin{thm}
\label{thm:product}There exists TD systems $(X,T)$ such that $(X\times X,T\times T)$
is not TD.
\end{thm}
On the other hand,
\begin{prop}
\label{pro:positive-powers}If $(X\times X,T\times T)$ is TD then
$(X,T^{n})$ is TD for all $n\geq1$.
\end{prop}
It is not clear as yet whether determinism of $(X,T)$ implies the
same for $(X,T^{n})$, $n\geq1$, although the converse is trivially
true, i.e. determinism of $(X,T^{n})$ for any $n>1$ implies it for
$(X,T)$.

In the next section we prove the proposition. In sections \ref{sec:Proof-of-inverse},
\ref{sec:proof-of-product} we give the constructions which prove
theorems \ref{thm:inverse}, \ref{thm:product}, respectively.

\section{\label{sec:Basic-properties}Basic properties of TD systems}

For general background on topological dynamics see e.g. \cite{Walters82}.
Given a system $(X,T)$ and $x\in X$ we write \[
\omega_{T}(x)=\bigcap_{n=1}^{\infty}\overline{\bigcup_{k\geq n}T^{k}x}\]

Let $T\times T$ denote the diagonal map on $X\times X$: i.e., $T\times T(x',x'')=(Tx',Tx'')$.
Let $CER(X)$ denote the space of closed equivalence relations on
$X$, and $ICER(X)$ for the invariant ones, i.e. \[
ICER(X)=\{R\in CER(X)\,:\, T\times T(R)=R\}\]
Also write $ICER^{+}(X)$ for the forward invariance equivalence relations:
\[
ICER^{+}(X)=\{R\in CER(X)\,:\, T\times T(R)\subseteq R\}\]
There is a bijection between factors of $(X,T)$ and members of $ICER^{+}(X)$,
given by the partition induced by the factor map. The image system
is invertable if and only if the corresponding relation is in $ICER(X)$.
It follows that \cite{KaminskiSiemaszkoSzymanski03}:
\begin{prop}
\label{pro:ICER-characterization}$(X,T)$ is TD if and only if $ICER^{+}(X)=ICER(X)$.
\end{prop}
A point $x\in X$ is forward recurrent if there is a sequence $n_{k}\rightarrow\infty$
such that $T^{n_{k}}x\rightarrow x$. Clearly if every point in $X\times X$
is $T\times T$ forward-recurrent then every forward invariant subset
of $X\times X$ is invariant, and in particular $ICER^{+}(X)=ICER(X)$.
This implies:
\begin{lem}
\label{lem:recurrence-in-XxX-implies-TD}Let $(X,T)$ be a topological
dynamical system. If every point of $X\times X$ is forward-recurrent
for $T\times T$ then $(X,T)$ is TD.
\end{lem}
This is the main condition used to establish that a system is TD.
We shall see that it is not in fact equivalent to TD, see Section
\ref{sec:proof-of-product}. However, there is a partial converse:
\begin{lem}
\label{lem:TD-implies-recurrence}If $(X,T)$ is deterministic then
every point in $X$ is forward recurrent for $T$.\end{lem}
\begin{proof}
Suppose $x\in X$ is not forward recurrent. Set \[
X_{0}=\{T^{n}x\,:\, n\geq0\}\cup\omega_{T}(x)\]
It is easily checked that $X_{0}$ is a closed and forward-invariant
but not invariant subset of $X$. Let \[
R=\{(x',x'')\,:\, x',x''\in X_{0}\}\cup\{(x,x)\,:\, x\in X\}\]
Then $R\in ICER^{+}$ but $R\notin ICER$. Hence $(X,T)$ is not TD.\end{proof}
\begin{lem}
\label{lem:forward-recurrence-of-powers}If $x$ is forward recurrent
for $T$ then $x$ is forward recurrent for $T^{n}$ for every $n\geq0$.\end{lem}
\begin{proof}
Denote by $\omega_{f}(y)$ the $\omega$-limit set of a point $y$
under a map $f$. Assuming the contrary, let $N$ be the least natural
number such that $x$ is not forward recurrent for $(X,T^{N})$, i.e.
$x\notin\omega_{T^{N}}(x)$ but $x\in\omega_{T^{n}}(x)$ for all $1\leq n<N$.
Since \[
\omega_{T}(x)=\bigcup_{k=0}^{N-1}\omega_{T^{N}}(T^{k}x)\]
there is some $0<r<N$ for which $x\in\omega_{T^{N}}(T^{r}x)$, or
equivalently $T^{M}x\in\omega_{T^{N}}(x)$, where $M=N-r$. Hence
$\omega_{T^{N}}(T^{M}x)\subseteq\omega_{T^{N}}(x)$. Since $T^{M}$
is an endomorphism of $(X,T)$, it follows from $T^{M}x\in\omega_{T^{N}}(x)$
that \[
T^{2M}x=T^{M}(T^{M}x)\in\omega_{T^{N}}(T^{M}x)\subseteq\omega_{T^{N}}(x)\]
and by induction $T^{kM}x\in\omega_{T^{N}}(x)$ for every $k\geq0$,
so $\omega_{T^{M}}(x)\subseteq\omega_{T^{N}}(x)$. Hence $x\notin\omega_{T^{M}}(x)$.
But $0<M<N$, contradicting the definition of $N$.
\end{proof}

\begin{proof}
[Proof of Proposition \ref{pro:positive-powers}] Suppose $(X\times X,T\times T)$
is TD; we wish to show that $(X,T^{n})$ is TD for all $n\geq1$. 

If $(X\times X,T\times T)$ is TD then, by \ref{lem:TD-implies-recurrence},
every point in $X\times X$ is forward recurrent for $T$. Hence,
by the last lemma, for every $n\geq1$, every point in $X\times X$
is forward recurrent for $(T\times T)^{n}$. Thus by Lemma \ref{lem:recurrence-in-XxX-implies-TD},
$(X,T^{n})$ is deterministic.
\end{proof}

\section{\label{sec:Proof-of-inverse}Proof of Theorem \ref{thm:inverse}}

We construct a deterministic system $(X,T)$ such that $(X,T^{-1})$
is not deterministic. 

A system $(X,T)$ is pointwise rigid if there exists a sequence $(n_{k})_{k=1}^{\infty}\subseteq\mathbb{N}$
such that $T^{n_{k}}x\rightarrow x$ for every $x\in X$. Clearly
this implies that $(X\times X,T\times T)$ is also pointwise rigid
and that every point in $X\times X$ is forward recurrent, so by Lemma
\ref{lem:recurrence-in-XxX-implies-TD} $(X,T)$ is TD. We shall construct
a pointwise rigid system such that $(X,T^{-1})$ contains a fixed
point $x_{0}$ and a point $x_{0}\neq x\in X$ such that $T^{-n}x\rightarrow x_{0}$;
thus $x$ is not forward recurrent for $T^{-1}$ so $(X,T^{-1})$
is not deterministic. Note that this also shows that $(X,T^{-1})$
is not pointwise rigid, even though $(X,T)$ is. A similar construction
appears in \cite{AuslanderGlasnerWeiss2007}.

Write $I=[0,1]$. Let $\mathbb{=N}=\{1,2,\ldots\}$ and endow $I^{\mathbb{N}}$
with the product topology. Write $x(i)$ for the $i$-th coordinate
of $x\in I^{\mathbb{N}}$ and let $T$ denote the shift map on $I^{\mathbb{N}},$
i.e. $(Tx)(i)=x(i+1)$.

We aim to construct a point $x\in I^{\mathbb{N}}$ and a sequence
$(n_{k})_{k=1}^{\infty}$, $n_{k}\rightarrow\infty$, such that 
\begin{enumerate}
\item $0^{k}1$ appears in $x$ for arbitrarily large $k$,
\item If $ab_{1},\ldots,b_{k+1}$ appears in $x$ for some symbols $a,b_{i}\in[0,1]$
and $b_{i}\leq\varepsilon$ for $i=1,\ldots,k$ then $a\leq\varepsilon+\frac{1}{k}$,
\item If $y=T^{m}x$ and $y(1)\ldots y(k)\neq0\ldots0$ then $|T^{n_{k}}y(i)-y(i)|<1/k$
for $i=1,\ldots,k$.
\end{enumerate}
Assuming we have constructed such a point $x$, take $X\subseteq[0,1]^{\mathbb{Z}}$
to be the bilateral extension of the orbit closure of $x$, that is,
the set of $y\in I^{\mathbb{Z}}$ such that every finite subword of
$y$ appears in some accumulation point of $\{T^{k}x\}_{k=1}^{\infty}$.
Condition (1) implies that the fixed point $\overline{0}=\ldots000\ldots$
is in $X$ and that there is a point $y=\ldots0001y'$ in $X$ for
some $y'\in I^{\mathbb{N}}$. Clearly the backward orbit of $y$ under
the shift converges to $\overline{0}$. Condition (3) implies that
if $z\in X$ is not forward-asymptotic to $\overline{0}$ then $T^{n(k)}z\rightarrow z$.
Finally, (2) guarantees that the only point which is forward asymptotic
to $\overline{0}$ is $\overline{0}$ itself: indeed, if $z$ is asymptotic
to $\underline{0}$ then, for every $\varepsilon>0$, there is an
$i_{0}$ such that $z(i)<\varepsilon$ for every $i>i_{0}$, and it
follows from this that $z(i)\leq\varepsilon$ for every $i\leq i_{0}$
as well, and consequently $z=\overline{0}$. Since $\overline{0}$
is a fixed point, (1)-(3) imply that $(X,T)$ is pointwise rigid. 

The definition of $x$ is by induction. Start the induction with $n_{1}=3$
and $x^{1}=100$. 

At the $m$-th stage of the construction we will have defined $n_{1},\ldots,n_{m}\in\mathbb{N}$
and $x^{m}=x(1)\ldots x(n_{m})$ and the final $m+1$ letters of $x^{m}$
will be $0$. 

Suppose this is the case; we must define $n_{m+1}$ and $x^{m+1}$.
For $t\in[0,1]$ let $t\cdot x^{m}$ for the pointwise product, i.e.
$(t\cdot x^{m})(i)=t\cdot x^{m}(i)$. Note that $0\cdot x^{m}=00\ldots0$.
Also write $ab$ for the concatenation of $a$ and $b$. Define \[
x^{m+1}=x^{m}x^{m}(\frac{m}{m+1}\cdot x^{m})(\frac{m-1}{m+1}\cdot x^{m})\ldots(\frac{1}{m+1}\cdot x^{m})(0\cdot x^{m})\]
and let $n_{m+1}$ be the length of $x^{m+1}$ (so by induction $n+m+1=(m+2)n_{m}$,
and in particular $n_{m}\geq m$). 

Each $x^{m}$ thus begins with a $1$ and ends with $0^{n_{m}}$,
and since $x^{m+1}$ begins with $x^{m}x^{m}$ condition (1) of the
construction holds. 

To verify (2), proceed by induction. It holds for subwords of $x^{1}$.
Suppose $ab_{1}\ldots b_{k+1}$ belongs to $x^{n+1}$. If $ab_{1}\ldots b_{k+1}$
belongs to one of the $t\cdot x^{n}$'s from which $x^{n+1}$ is constructed
then we are done by the induction hypothesis. Otherwise one of the
$b_{i}$'s is the first symbol of one of the $tx^{n}$'s. Let $b_{i}$
be the first of these and $t=\frac{r}{n+1}$; the fact that $b_{i}<\varepsilon$
means that $\frac{r}{n+1}<\varepsilon$. Hence $a$ belongs to $\frac{r+1}{n+1}x^{n}$,
so $a\leq\frac{r+1}{n+1}\leq\varepsilon+\frac{1}{n+1}$. 

For (3), we claim that for each $m$ and $k<m$ if $0\leq i<n_{m}-n_{k}$
and $x^{m}(i),\ldots,x^{m}(i+k-1)\neq0$ then $|x^{m}(i)-x^{m}(i+n_{k}-1)|<1/k$.
The proof is by induction on $m$, using the fact that if $y$ satisfies
this condition then so does $t\cdot y$ for $t\in[0,1]$. Specifically,
let $m,k,i$ as above. If $k=m-1$ the proof is immediate from the
construction. Otherwise write $x^{m}=y_{1}\ldots y_{m+2}$ with $y_{j}=t_{j}x^{m-1}$
as in the definition. Let $i=s\cdot n_{m-1}+i'$ for $s,i'\in\{0,1,\ldots,n_{m-1}-1\}$.
If $0\leq i'\leq n(m-1)-k$ we can apply the induction hypothesis.
Otherwise, $i'$ is in the final $0^{n_{m-2}}$-block of $y_{s}$
so the assumption that $x^{m}(i),\ldots,x^{m}(i+n_{k}-1)\neq0$ implies
that $i'>n_{m-1}-k$. But now note that $y_{s+1}=x^{k}z$ for some
$z$, so $y_{s+1}(n_{k}-i')=0$ because the final $k$ letters of
$x^{k}$ are $0$. So $x^{m}(i)=x^{m}(i+n_{k})=0$ and we are done.

\section{\label{sec:proof-of-product}Proof of Theorem \ref{thm:product}}

We shall construct a system $(X,T)$ which is TD, but $(X\times X,T\times T)$
is not TD. To establish the first property, we rely on the following
result:
\begin{lem}
Suppose $(X,T)$ has the property that for every $(x',x'')\in X\times X$,
either $(x',x'')$ is forward recurrent for $T\times T$ or else there
is a $p\in X$ such that $(x',p),(p,x'')\in\omega_{T\times T}(x',x'')$.
Then $(X,T)$ is deterministic.\end{lem}
\begin{proof}
It suffices to show that $ICER^{+}=ICER$. Let $R\in ICER^{+}$ and
let $(x',x'')\in R$. Since $\omega_{T\times T}(x',x'')\subseteq TR$,
if the first condition holds (i.e. if $(x',x'')\in\omega_{T\times T}(x',x'')$)
then $(x',x'')\in TR$. Otherwise there is a $p\in X$ so that $(x',p),(p,x'')\in\omega_{T\times T}(x',x'')\subseteq TR$,
and since $TR$ is an equivalence relation, this means $(x',x'')\in TR$.
We have shown that $(x',x'')\in TR$ whenever $(x',x'')\in R$, so
$R\subseteq TR$. The reverse containment holds by assumption so $R\in ICER$,
and the lemma follows.
\end{proof}
We shall construct a system containing a fixed point which will play
role of the point $p$ in the lemma, i.e. every pair $(x',x'')$ in
the system which is not forward recurrent will have $(x',p),(p,x'')\in\omega_{T\times T}(x',x'')$.
For simplicity we describe a non-transitive example, and then explain
how to modify it to get a transitive one.

Let $T$ be the shift on $[0,1]^{\mathbb{Z}}$. A block is a finite
subsequence $x\in[0,1]^{\{1,\ldots,n\}}$; here $n$ is the length
of the block. If $x,y$ are blocks of length $m,n$ respectively their
concatenation is written $xy$ and is the block $x(1)\ldots x(m)y(1)\ldots y(n)$
of length $m+n$. For $x\in[0,1]^{\mathbb{Z}}$ a sub-block is a block
of the form $x(i),x(i+1),\ldots,x(j)$; this is the block of length
$j-i+1$ appearing in $x$ at $i$. We denote this sub-block by $x(i;j)$.
We say that blocks $x_{1},x_{2}$ occur consecutively in $x$ if $x_{1}=x(i,j)$
and $x_{2}=x(j+1,k)$ for some $i\leq j<k$. 

To construct the example we define two points $x^{*},y^{*}\in[0,1]^{\mathbb{Z}}$
with $x^{*}(1)=y^{*}(1)=1$, and take $X,Y$ to be their orbit closure,
respectively. We also define sequences $m_{k}\rightarrow\infty$ and
$n_{k}\rightarrow\infty$ so that the following conditions are satisfied:
\setenumerate[1]{label=(\roman*)}
\begin{enumerate}
\item \label{cond:x-rigidity}$\Vert x^{*}-T^{m_{k}}x^{*}\Vert_{\infty}\leq\frac{1}{k}$
for $k\geq1$.
\item \label{cond:y-rigidity}$\Vert y^{*}-T^{n_{k}}y^{*}\Vert_{\infty}\leq\frac{1}{k}$
for $k\geq1$.
\item \label{cond:x-sparseness}For $k\geq1$, out of every three consecutive
blocks in $x^{*}$ of length $n_{k}$ at least two are identically
$0$. 
\item \label{cond:y-sparseness}For $k\geq1$, out of every three consecutive
blocks in $y^{*}$ of length $m_{k}$ at least two are identically
$0$. 
\item \label{cond:orthogonality}For every $k\neq0$, at least one of the
symbols $x^{*}(k)$ or $y^{*}(k)$ is equal to $0$.
\end{enumerate}
Let $X$ be the orbit closure of $x^{*}$ and $Y$ the orbit closure
of $y^{*}$. We claim that given such points $x^{*},y^{*}$ the system
$Z=X\cup Y$ is deterministic, but $Z\times Z$ is not. Indeed, the
latter statement follows from the observation that by condition \ref{cond:orthogonality}
and the fact that $x(0)=y(0)=1$, the pair $(x^{*},y^{*})\in Z\times Z$
is not forward recurrent for $T\times T$, so $Z\times Z$ is not
deterministic.

To see that $Z$ is deterministic, note that the properties \eqref{cond:x-rigidity}--\eqref{cond:y-sparseness}
above hold when $x^{*},y^{*}$ is replaced by any pair $x\in X,y\in Y$.
Condition \eqref{cond:x-rigidity} now implies that $T^{m_{k}}|_{X}\rightarrow\textrm{Id}_{X}$
uniformly, and similarly \eqref{cond:y-rigidity} implies that $T^{n_{k}}|_{Y}\rightarrow\textrm{Id}_{Y}$
uniformly, and in particular every pair in $X$ is forward recurrent
for $T\times T$ and so is every pair from $Y$. For $x\in X,y\in Y$,
conditions \eqref{cond:x-rigidity} and \eqref{cond:y-sparseness}
imply that there is a choice of $r(k)\in\{1,2,3\}$ so that $T^{r(k)m_{k}}x\rightarrow x$
but $T^{r(k)m_{k}}y\rightarrow\overline{0}$, and hence $(x,\overline{0})\in\omega_{T\times T}(x,y)$.
Similarly \eqref{cond:y-rigidity} and \eqref{cond:x-sparseness}
imply that there is a choice $s(k)\in\{1,2,3\}$ so that $T^{s(k)n_{k}}x\rightarrow\overline{0}$
but $T^{s(k)n_{k}}y\rightarrow y$, so also $(\overline{0},y)\in\omega_{T\times T}(x,y)$.
From the lemma it now follows that $Z=X\cup Y$ is deterministic.

Here are the details of the construction. We proceed by induction
on $r$. At the $r$-th stage we will be given an integer $L(r)\geq r-1$
and finite sequences $x_{r},y_{r}\in[0,1]^{\{-L(r),-L(r)+1,\ldots,L(r)\}}$,
and if $r\geq2$ we are also given integers $m_{r-1},n_{r-1}$ . We
extend $x_{r}$ to $x_{r+1}$ and $y_{r}$to $y_{r+1}$ without changing
the symbols already defined. The blocks $x_{r},y_{r}$ will satisfy
the following versions of the conditions above, and an additional
condition which is required for the induction: \setenumerate[1]{label=(\Roman*)}
\begin{enumerate}
\item \label{cond:finite-x-rigidity}$\Vert x_{r}(i;i+k)-x_{r}(i+m_{k};i+m_{k}+k)\Vert_{\infty}\leq\frac{1}{k}$
for $1\leq k\leq r-1$ and $-L(r)\leq i\leq L(r)-m_{k}-k$.
\item \label{cond:finite-y-rigidity}$\Vert y_{r}(i;i+k)-y_{r}(i+n_{k};i+n_{k}+k)\Vert_{\infty}\leq\frac{1}{k}$
for $1\leq k\leq r-1$ and $-L(r)\leq i\leq L(r)-n_{k}-k$.
\item \label{cond:finite-x-sparseness}For $1\leq k\leq r-1$, out of every
three consecutive blocks in $x_{r}$ of length $n_{k}$ at least two
are identically $0$. 
\item \label{cond:finite-y-sparseness}For $1\leq k\leq r-1$, out of every
three consecutive blocks in $y_{r}$ of length $m_{k}$ at least two
are identically $0$. 
\item \label{cond:finite-orthogonality}For every $k\neq0$ between $-L(r)$
and $L(r)$, at least one of the symbols $x_{r}(k)$ or $y_{r}(k)$
are equal to $0$.
\item \label{cond:zero-tails}$m_{k},n_{k}\leq L(r-1)$ for each $1\leq k\leq r-1$,
and the first and last $2L(r-1)$ symbols of $x_{r}$ and $y_{r}$
are $0$. 
\end{enumerate}
Assuming that such a sequence $x_{r},y_{r}$ exists, define $x^{*},y^{*}\in[0,1]^{\mathbb{Z}}$
by $x^{*}(i)=x_{i+1}(i)$ and $y^{*}(i)=y_{i+1}(i)$. It is straightforward
to verify that these conditions guarantee that $x^{*},y^{*}$have
the desired properties.

We start the induction by $L(1)=0$ and $x_{1}(0)=y_{1}(0)=1$; all
conditions are satisfied trivially.

For some $r\geq1$ suppose we are given $x_{r},y_{r},L(r)$ and also
$m_{k},n_{k}$ for $0\leq k<r$, such that \eqref{cond:finite-x-rigidity}-\eqref{cond:zero-tails}
are satisfied. For a block $z$ and $\alpha\in[0,1]$, denote by $\alpha\cdot z$
the block with $(\alpha z)(i)=\alpha\cdot z(i)$. 

Let $s,t,s',t'$ be integers which we shall specify later. Let $u$
and $v$ be blocks of $0$'s of length $s,t$, respectively, and set
\[
x_{r+1}=v\,(\frac{1}{r+1}\cdot x_{r})\, u\,\ldots\, u\,(\frac{r}{r+1}\cdot x_{r})\, u\, x_{r}\, u\,(\frac{r}{r+1}\cdot x_{r})\, u\,(\frac{r-1}{r+1}\cdot x_{r})\, u\,\ldots\, u\,(\frac{1}{r+1}\cdot x_{r})\, v\]
Let $u',v'$ to be blocks of $0$'s of length $s',t'$ respectively,
and set\[
y_{r+1}=v'\,(\frac{1}{r+1}\cdot y_{r})\, u'\ldots u'\,(\frac{r}{r+1}\cdot y_{r})\, u'\, y_{r}\, u'\,(\frac{r}{r+1}\cdot y_{r})\, u'\,(\frac{r-1}{r+1}\cdot y_{r})\, u'\ldots u\,(\frac{1}{r+1}\cdot y_{r})\, v'\]
Note that in defining $x_{r+1},y_{r+1}$ we have added blocks to the
left and right of the central copy of $x_{r},y_{r}$, respectively,
without changing the central blocks. We will assume that $s,t,s',t'$
are chosen so that the lengths of $x_{r+1},y_{r+1}$ are equal,. We
define $L(r+1)$ to be their common length. See figure \ref{fig:construction}.

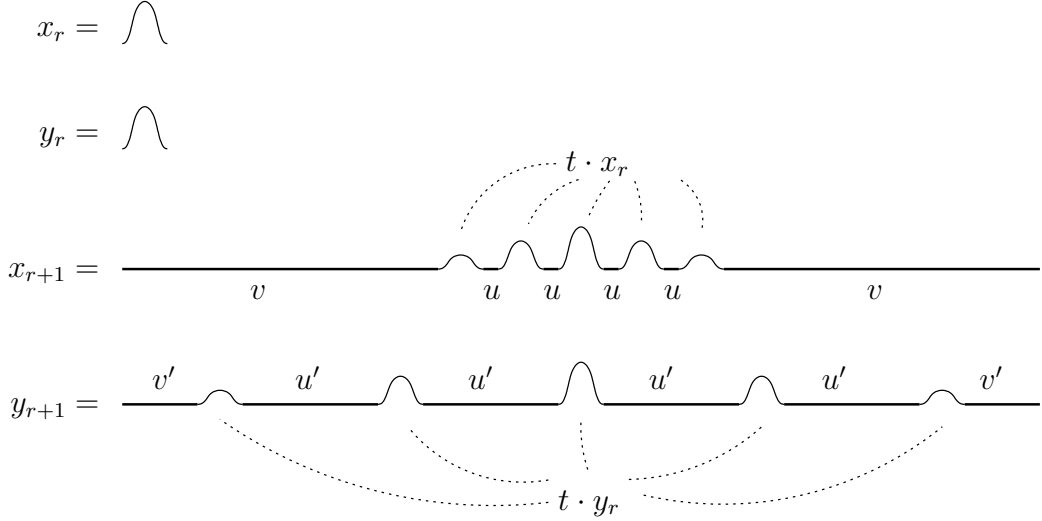
\begin{figure} \input{./construction.pstex_t} \caption{The construction of $x_{r+1},y_{r+1}$ form $x_r,y_r$} (schematic) \label{fig:construction} \end{figure}

By condition \eqref{cond:zero-tails}, $x_{r+1}$ and $y_{r+1}$ satisfy
\eqref{cond:finite-x-rigidity} and \eqref{cond:finite-y-rigidity}
for $r+1$ and $1\leq k<r$. More precisely, suppose that $1\leq k<r$
and $L(r+1)\leq i\leq L(r+1)-m_{k}+1$, and consider the blocks of
length $k$ in $x_{r+1}$ at positions $i$ and $i+m_{k}$. There
are two possibilities. Either both blocks are located inside the same
copy of $t\cdot x_{r}$ for some $t$, in which case $\Vert x_{r}(i;i+k)-x_{r}(i+m_{k};i+m_{k}+k)\Vert_{\infty}\leq\frac{1}{k}$
by the induction hypothesis, or else at least one is located in an
$u$ and the other either in the first or last $m_{r}$ symbols of
a block of the form $t\cdot x_{r}$. In both of the last possibilities,
the blocks are blocks of $0$'s (because $u$ is all $0$'s and because
of condition \eqref{cond:zero-tails} of the induction hypothesis)
so $\Vert x_{r}(i;i+k)-x_{r}(i+m_{k};i+m_{k}+k)\Vert_{\infty}\leq\frac{1}{k}$
is satisfied trivially. The analysis for $y_{r+1}$ is similar.

Define $m_{r}=L(r)+s$. Then $x_{r+1}$ also satisfies condition \eqref{cond:finite-x-rigidity}
for $k=r$, because every two symbols in $x_{r+1}$ whose distance
is $L(r)+s$ belong to blocks of the form $\frac{i}{r+1}\cdot x_{r}$
and $\frac{i\pm1}{r+1}\cdot x_{r}$, and so differ in value by at
most $\frac{1}{r+1}$. Similarly, if we define $n_{r}=L(r)+s'$ then
$y_{r+1}$ satisfies \eqref{cond:finite-y-rigidity} for $k=r$. 

If we choose $s,t,s',t'$ large enough, conditions \eqref{cond:finite-x-sparseness},\eqref{cond:finite-y-sparseness}
hold for $x_{r+1},y_{r+1}$. The same is true also for \eqref{cond:zero-tails}.

It remains to obtain \eqref{cond:finite-orthogonality}. We still
have freedom to choose $s,s',t,t'$ subject to the restriction that
$x_{r+1},y_{r+1}$ have the same length, and as long as they are large
enough. We first fix $s$ some arbitrarily sufficiently large number
(this determines the value of $m_{k}$). Next, we select $s'$ large
enough so that each non-zero component of $x_{r+1}$ is opposite the
central block $0^{s'}\, y_{r}\,0^{s'}$ in $y_{r+1}$ (here $0^{m}$
is the word consisting of $m$ zeros); this implies also that each
non-zero symbol in $y_{r+1}$ outside of the central block $y_{r}$
is opposite a $0$ in $x_{r+1}$. This and the induction hypothesis
guarantees that \ref{cond:finite-orthogonality} holds. It remains
only to note that although $t$ determines $t'$, we can still make
each as large as we want. This completes the construction.

To give a transitive example, one adds an intermediate step between
each step of the construction above. Given $x_{r},y_{r}$ one forms
the blocks\begin{eqnarray*}
x'_{r} & = & by_{r}ax_{r}ay_{r}b\\
y'_{r} & = & dx_{r}cy_{r}cx_{r}d\end{eqnarray*}
where $a,b,c,d$ are sufficiently long blocks of $0$'s chosen so
that $x'_{r},y'_{r}$ have the same length $L'(r)$ and condition
\eqref{cond:finite-orthogonality} holds for $x'_{r},y'_{r}$. Now
carry out the induction step above obtaining $x_{r+1},y_{r+1}$ from
$x'_{r},y'_{r}$. Conditions \eqref{cond:finite-x-rigidity},\eqref{cond:finite-y-rigidity}
no longer hold but a modified version does, in which we replace given
a block of length $1\leq k\leq r-1$ in $x_{r}$ or $y_{r}$, it repeats
with accuracy $1/k$ at distance either $m_{k}$ or $n_{k}$. The
points $x^{*},y^{*}$ will now be transitive for $Z$, and an argument
similar to the above will show that $Z$ is deterministic but $Z\times Z$
is not. 

Finally, note that not every point in $X\times X$ is forward recurrent
but $X$ is TD. This shows that Lemma \ref{lem:recurrence-in-XxX-implies-TD}
is only a sufficient condition for TD, not necessary condition.

\bibliographystyle{plain}
\bibliography{bib}

\end{document}

%% file: construction.pstex_t
\begin{picture}(0,0)%
\includegraphics{construction.pstex}%
\end{picture}%
\setlength{\unitlength}{4144sp}%
\begingroup\makeatletter\ifx\SetFigFont\undefined%
\gdef\SetFigFont#1#2#3#4#5{%
  \reset@font\fontsize{#1}{#2pt}%
  \fontfamily{#3}\fontseries{#4}\fontshape{#5}%
  \selectfont}%
\fi\endgroup%
\begin{picture}(6570,3112)(-67,-3359)
\put(3646,-1276){\makebox(0,0)[lb]{\smash{{\SetFigFont{12}{14.4}{\familydefault}{\mddefault}{\updefault}{\color[rgb]{0,0,0}$t\cdot{}x_r$}%
}}}}
\put(1171,-2581){\makebox(0,0)[lb]{\smash{{\SetFigFont{12}{14.4}{\familydefault}{\mddefault}{\updefault}{\color[rgb]{0,0,0}$v'$}%
}}}}
\put(6121,-2581){\makebox(0,0)[lb]{\smash{{\SetFigFont{12}{14.4}{\familydefault}{\mddefault}{\updefault}{\color[rgb]{0,0,0}$v'$}%
}}}}
\put(2026,-2581){\makebox(0,0)[lb]{\smash{{\SetFigFont{12}{14.4}{\familydefault}{\mddefault}{\updefault}{\color[rgb]{0,0,0}$u'$}%
}}}}
\put(3061,-2581){\makebox(0,0)[lb]{\smash{{\SetFigFont{12}{14.4}{\familydefault}{\mddefault}{\updefault}{\color[rgb]{0,0,0}$u'$}%
}}}}
\put(4141,-2581){\makebox(0,0)[lb]{\smash{{\SetFigFont{12}{14.4}{\familydefault}{\mddefault}{\updefault}{\color[rgb]{0,0,0}$u'$}%
}}}}
\put(5176,-2581){\makebox(0,0)[lb]{\smash{{\SetFigFont{12}{14.4}{\familydefault}{\mddefault}{\updefault}{\color[rgb]{0,0,0}$u'$}%
}}}}
\put(3601,-3301){\makebox(0,0)[lb]{\smash{{\SetFigFont{12}{14.4}{\familydefault}{\mddefault}{\updefault}{\color[rgb]{0,0,0}$t\cdot{}y_r$}%
}}}}
\put(856,-466){\makebox(0,0)[rb]{\smash{{\SetFigFont{12}{14.4}{\familydefault}{\mddefault}{\updefault}{\color[rgb]{0,0,0}$x_r=$}%
}}}}
\put(856,-1096){\makebox(0,0)[rb]{\smash{{\SetFigFont{12}{14.4}{\familydefault}{\mddefault}{\updefault}{\color[rgb]{0,0,0}$y_r=$}%
}}}}
\put(856,-1906){\makebox(0,0)[rb]{\smash{{\SetFigFont{12}{14.4}{\familydefault}{\mddefault}{\updefault}{\color[rgb]{0,0,0}$x_{r+1}=$}%
}}}}
\put(856,-2716){\makebox(0,0)[rb]{\smash{{\SetFigFont{12}{14.4}{\familydefault}{\mddefault}{\updefault}{\color[rgb]{0,0,0}$y_{r+1}=$}%
}}}}
\put(3151,-2041){\makebox(0,0)[lb]{\smash{{\SetFigFont{12}{14.4}{\familydefault}{\mddefault}{\updefault}{\color[rgb]{0,0,0}$u$}%
}}}}
\put(3511,-2041){\makebox(0,0)[lb]{\smash{{\SetFigFont{12}{14.4}{\familydefault}{\mddefault}{\updefault}{\color[rgb]{0,0,0}$u$}%
}}}}
\put(3871,-2041){\makebox(0,0)[lb]{\smash{{\SetFigFont{12}{14.4}{\familydefault}{\mddefault}{\updefault}{\color[rgb]{0,0,0}$u$}%
}}}}
\put(4231,-2041){\makebox(0,0)[lb]{\smash{{\SetFigFont{12}{14.4}{\familydefault}{\mddefault}{\updefault}{\color[rgb]{0,0,0}$u$}%
}}}}
\put(5446,-2041){\makebox(0,0)[lb]{\smash{{\SetFigFont{12}{14.4}{\familydefault}{\mddefault}{\updefault}{\color[rgb]{0,0,0}$v$}%
}}}}
\put(1756,-2041){\makebox(0,0)[lb]{\smash{{\SetFigFont{12}{14.4}{\familydefault}{\mddefault}{\updefault}{\color[rgb]{0,0,0}$v$}%
}}}}
\end{picture}%